\documentclass[a4paper,12pt,reqno]{amsart}
\usepackage[english]{babel}
\usepackage[T1]{fontenc}
\usepackage[left=1.1in,right=1.1in,top=1.2in,bottom=1.2in]{geometry}
\usepackage{times}
\usepackage{microtype}

\usepackage{commath,amssymb,amscd,mathrsfs,mathtools,dsfont,bbm}
\usepackage[all]{xy}
\usepackage{enumerate}
\usepackage{longtable}

\usepackage[usenames,dvipsnames,svgnames,table]{xcolor}
\definecolor{red}{RGB}{255,25,25}
\definecolor{blue}{RGB}{25,50,200}
\usepackage{tikz-cd}
\usetikzlibrary{decorations.pathmorphing}

\usepackage[pagebackref,linktocpage]{hyperref}

\hypersetup{
pdftitle={},			% title
pdfauthor={},		% author
colorlinks=true,				% false: boxed links; true: colored links
linkcolor=red,				% color of internal links (change box color with linkbordercolor)
citecolor=MidnightBlue,		% color of links to bibliography
filecolor=magenta,			% color of file links
urlcolor=MidnightBlue			% color of external links
}

\usepackage{cleveref} % after hyperref!!

\newtheorem{theorem}{Theorem}[section]
\crefname{theorem}{Theorem}{Theorems}
\newtheorem{lemma}[theorem]{Lemma}
\crefname{lemma}{Lemma}{Lemmas}

\crefname{proposition}{Proposition}{Propositions}

\crefname{prop}{Proposition}{Propositions}

\crefname{corollary}{Corollary}{Corollaries}

\crefname{cor}{Corollary}{Corollaries}

\crefname{conjecture}{Conjecture}{Conjectures}

\crefname{conj}{Conjecture}{Conjectures}
\newtheorem*{conj*}{Conjecture}
\crefname{conj}{Conjecture}{Conjectures}

\theoremstyle{definition}
\newtheorem{definition}[theorem]{Definition}
\crefname{definition}{Definition}{Definitions}

\crefname{defn}{Definition}{Definitions}

\crefname{example}{Example}{Examples}

\crefname{notation}{Notation}{Notation}
\newtheorem*{notation*}{Notation}
\crefname{notation}{Notation}{Notation}

\crefname{problem}{Problem}{Problems}

\crefname{question}{Question}{Questions}

\crefname{condition}{Condition}{Conditions}

\crefname{assumption}{Assumption}{Assumptions}

\newtheorem{hyp}{Hyp}
\crefname{hyp}{Hyp}{Hyp}

\theoremstyle{remark}
\newtheorem{rmk}[theorem]{Remark}
\crefname{rmk}{Remark}{Remarks}
\newtheorem*{rmk*}{Remark}
\crefname{rmk}{Remark}{Remarks}
\newtheorem{remark}[theorem]{Remark}
\crefname{remark}{Remark}{Remarks}

\crefname{fact}{Fact}{Facts}

\crefname{claim}{Claim}{Claims}
\newtheorem*{claim*}{Claim}
\crefname{claim}{Claim}{Claims}

\crefname{step}{Step}{Steps}

\crefname{case}{Case}{Cases}

\numberwithin{equation}{section}
\newcommand{\ratmap}{\dashrightarrow}

\newcommand{\lra}{\longrightarrow}

\newcommand{\arxiv}[1]{\href{https://arxiv.org/abs/#1}{{\tt arXiv:#1}}}
\newcommand{\doi}[1]{\href{https://doi.org/#1}{{\tt DOI:#1}}}
\def\MR#1{\href{https://www.ams.org/mathscinet-getitem?mr=#1}{MR~#1~}}

\newcommand{\bbP}{\mathbb{P}}

\newcommand{\bC}{\mathbf{C}}

\newcommand{\bF}{\mathbf{F}}

\newcommand{\bQ}{\mathbf{Q}}
\newcommand{\bR}{\mathbf{R}}
\newcommand{\bZ}{\mathbf{Z}}

\newcommand{\sO}{\mathscr{O}}

\newcommand{\reg}{{\rm reg}}

\newcommand{\GL}{\mathrm{GL}}

\newcommand{\alb}{\operatorname{alb}}
\newcommand{\Alb}{\operatorname{Alb}}

\newcommand{\Aut}{\operatorname{Aut}}

\newcommand{\dr}{\operatorname{dr}}

\newcommand{\id}{\operatorname{id}}

\newcommand{\isom}{\simeq}

\newcommand{\NS}{\operatorname{NS}}
\newcommand{\num}{\equiv}

\newcommand{\Pic}{\operatorname{Pic}}

\newcommand{\PsAut}{\operatorname{PsAut}}

\newcommand{\rank}{\operatorname{rank}}

\begin{document}

\title[Free abelian groups of sub-maximal dynamical rank]{Free abelian group actions on normal projective varieties: sub-maximal dynamical rank case}

\author{Fei Hu}
\address{University of British Columbia, Vancouver, BC V6T 1Z2, Canada \endgraf
Pacific Institute for the Mathematical Sciences, Vancouver, BC V6T 1Z4, Canada}
\curraddr{University of Waterloo, Waterloo, ON N2L 3G1, Canada}
\email{\href{mailto:hf@u.nus.edu}{hf@u.nus.edu}}
\urladdr{\url{https://sites.google.com/view/feihu90s/}}

\author{Sichen Li}
\address{East China Normal University, 500 Dongchuan Road, Shanghai 200241, China \endgraf
National University of Singapore, 10 Lower Kent Ridge Road, Singapore 119076}
\email{\href{mailto:lisichen123@foxmail.com}{lisichen123@foxmail.com}}

\begin{abstract}
Let $X$ be a normal projective variety of dimension $n$ and $G$ an abelian group of automorphisms such that all elements of $G\setminus \{\id\}$ are of positive entropy.
Dinh and Sibony showed that $G$ is actually free abelian of rank $\le n - 1$.
The maximal rank case has been well understood by De-Qi Zhang.
We aim to characterize the pair $(X, G)$ such that $\rank G = n - 2$.
\end{abstract}


%\date{\today}


\thanks{The first author was partially supported by postdoctoral fellowships of the University of British Columbia, the Pacific Institute for the Mathematical Sciences, and the University of Waterloo.
The second author was supported by a scholarship of China Scholarship Council.}

\maketitle

%%%%%%%%%%%%%%%%%%%%%%%%%%%%%%%%%%%%%%%%%%%%%%%%%%%%%%%%%%

\section{Introduction}
\label{sec:intro}

%%%%%%%%%%%%%%%%%%%%%%%%%%%%%%%%%%%%%%%%%%%%%%%%%%%%%%%%%%

\noindent
We work over the field $\bC$ of complex numbers.
Let $X$ be a normal projective variety.
Denote by $\NS(X) \coloneqq \Pic(X)/\Pic^{\circ}(X)$ the {\it N\'eron--Severi group} of $X$, i.e., the finitely generated abelian group of Cartier divisors on $X$ modulo algebraic equivalence.
For a field $\bF = \bQ, \bR$ or $\bC$, we denote by $\NS(X)_\bF$ the finite-dimensional $\bF$-vector space
$\NS(X) \otimes_{\bZ} \bF$.
The {\it first dynamical degree $d_1(g)$} of an automorphism $g \in \Aut(X)$ is defined as the spectral radius of its natural pullback $g^*$ on $\NS(X)_\bR$, i.e.,
\[
d_1(g) \coloneqq \rho\Big(g^*|_{\NS(X)_\bR} \Big) = \max\Big\{|\lambda| : \lambda \text{ is an eigenvalue of } g^*|_{\NS(X)_\bR} \Big\}.
\]

We say that $g$ is of {\it positive entropy} if $d_1(g)>1$, otherwise it is of {\it null entropy}.
For a subgroup $G$ of the automorphism group $\Aut(X)$, we define the {\it null-entropy subset of $G$} as
\[
N(G) \coloneqq \left\{ g \in G : g \text{ is of null entropy, i.e., } d_1(g) = 1 \right\}.
\]
We call $G$ of {\it positive entropy} (resp. of {\it null entropy}), if $N(G) = \{\id\} \subsetneq G$ (resp. $N(G) = G$).
Indeed, when $X$ is smooth and hence $X(\bC)$ is a compact K\"ahler manifold, our positivity notion of entropy is equivalent to the positivity of topological entropy in complex dynamics by the log-concavity of dynamical degrees and the fundamental work of Gromov \cite{Gromov03} and Yomdin \cite{Yomdin87}.
We refer to \cite[\S4]{DS17} and references therein for a comprehensive exposition on dynamical degrees, topological and algebraic entropies.

In \cite{DS04}, Dinh and Sibony proved that for any abelian subgroup $G$ of $\Aut(X)$, if $G$ is of positive entropy, then $G$ is free abelian of rank $\le \dim X-1$.
This was subsequently extended by De-Qi Zhang \cite{Zhang09} to the solvable group case.
We are thus interested in algebraic varieties admitting the action of free abelian groups of positive entropy.
Therefore, it is meaningful for us to consider the following hypothesis.

\begin{hyp}
\label{hyp}
$X$ is a normal projective variety of dimension $n$ and $G \isom \bZ^{r}$ is a subgroup of $\Aut(X)$ with $1 \le r \le n-1$, such that $G$ is of positive entropy, i.e., all elements of $G \setminus \{\id\}$ are of positive entropy.
\end{hyp}

Often, we shall call the above positive integer $r$ the {\it dynamical rank} of $G$ to emphasize that $G$ is of positive entropy in the context of dynamics, not just being a free abelian group.
See \cref{subsec:dyn-rank} for a more general consideration on dynamical ranks.

In the last years, the maximal dynamical rank case $r = n-1$ has been intensively studied by De-Qi Zhang in his series papers (see e.g., \cite{Zhang09,Zhang13,Zhang16}), which extend the known surface case \cite{Cantat99} to higher dimensions.
See also \cite{DS04,Oguiso07,CZ12,CWZ14,DHZ15,OT15,Lesieutre18} for relevant work.
We rephrase one of Zhang's main results as follows.

\begin{theorem}[{cf.~\cite[Theorems 1.1 and 2.4]{Zhang16}}]
\label{thm:Zhang}
Let $(X,G)$ satisfy \hyperref[hyp]{\rm Hyp$(n,n-1)$} with $n \ge 3$.
Suppose that $X$ is not rationally connected, or $X$ has only $\bQ$-factorial Kawamata log terminal (klt) singularities and the canonical divisor $K_X$ is pseudo-effective.
Then after replacing $G$ by a finite-index subgroup, the following assertions hold.
\begin{itemize}
\item[(1)] There is a birational map $X \ratmap Y = A/F$ such that the induced action of $G$ on $Y$ is biregular, where $A$ is an abelian variety and $F$ is a finite group whose action on $A$ is free outside a finite subset of $A$.
\item[(2)] The canonical divisor of $Y$ is $\bQ$-linearly equivalent to zero, i.e., $K_Y \sim_\bQ 0$.
\item[(3)] There is a faithful action of $G$ on $A$ such that $A \lra A/F=Y$ is $G$-equivariant. Every $G$-periodic proper subvariety of $Y$ or $A$ is a point.
\end{itemize}
\end{theorem}

Our hypothesis \hyperref[hyp]{\rm Hyp$(n,n-1)$} is nothing but Zhang's Hyp(sA), which is stronger than his Hyp(A). But the latter is a property preserved by generically finite maps.
Note that one of the key ingredients of Zhang's proof is the existence of certain $G$-equivariant log minimal model program (or rather, LMMP with scaling), where the klt singularity assumption has its significance (see \cite[Lemma~3.13]{Zhang16}).
On the other hand, for ease of exposition, we are blindly using klt singularity other than log terminal singularity, though there is no actual pair but just $X$ so that they are actually the same.

The aim of this article is to investigate the sub-maximal dynamical rank case $r = n-2$ following Nakayama and Zhang's ideas in \cite{NZ10,Nakayama10}.
Although they only dealt with polarized endomorphisms of normal projective varieties, the machinery developed there is robust so that it could also be adopted in the study of automorphisms.
We refer to \cref{subsec:weak-decomp,subsec:SMRC} for their counterparts.
Other ingredients include the product formula of dynamical degrees due to Dinh and Nguy\^en \cite{DN11} and an inequality about dynamical ranks given by the first author in \cite{Hu19} (see \cref{Zha-Hu1}).
Below is our main result.

\begin{theorem}
\label{thmA}
Let $(X,G)$ satisfy \hyperref[hyp]{\rm Hyp$(n,n-2)$} with $n \ge 3$.
Then the Kodaira dimension $\kappa(X)$ of $X$ is at most one.
Moreover, after replacing $G$ by a finite-index subgroup, we obtain the following partial classification.
\begin{itemize}
\item[(1)] When $\kappa(X) = 1$, let $F$ be a very general fiber of the Iitaka fibration $X \ratmap B$ of $X$, where $\dim B = 1$. Then $G$ descends to a trivial action on the base curve $B$ and acts faithfully on $F$ such that $F$ is $G$-equivariantly birational to a K3 surface, an Enriques surface, or a Q-abelian variety (see \cref{def:Q-abel}).

\item[(2)] When $\kappa(X) = 0$, suppose further that $X$ has only klt singularities and $K_X \num 0$.
Then there exists a finite cover $Y \lra X$, \'etale in codimension one, such that $Y$ is $G$-equivariantly birational to a weak Calabi--Yau variety (see \cref{def:wCY}), an abelian variety, or a product of a weak Calabi--Yau surface and an abelian variety.

\item[(3)] When $\kappa(X) = -\infty$, suppose further that $X$ is uniruled.
Let $\pi \colon X \ratmap Z$ be the special MRC fibration of $X$ (in the sense of Nakayama; see e.g., \cref{def:sMRC}).
Then either $X$ is rationally connected, or $Z$ is birational to a curve of genus $\ge 1$, a K3 surface, an Enriques surface, or a Q-abelian variety $A/F$, where $A$ is an abelian variety and $F$ is a finite group whose action on $A$ is free outside a finite subset of $A$. In particular, if $\dim Z\ge 3$, then there exists a finite cover $X' \lra X$, \'etale in codimension one, such that the induced rational map $\pi' \colon X' \ratmap A$ is $G$-equivariantly birational to the MRC fibration of $X'$.
\end{itemize}
\end{theorem}

\begin{rmk}
\label{rmkA}
\begin{enumerate}[(1)]
\item %In the case $\kappa(X) = 0$, the klt singularity assumption is used in an Abundance theorem by Nakayama \cite{Nakayama04} to reduce $K_X \num 0$ to $K_X \sim_\bQ 0$; see \cref{rmk:Nakayama}(1).
In the case $\kappa(X)=0$, if we merely assume that $X$ has only klt singularities, then the good minimal model program predicts the existence of a minimal model $X_m$ of $X$ so that $K_{X_m} \sim_\bQ 0$.
Modulo this, one then has to consider the induced birational (not necessarily biregular) action of $G$ on $X_m$.
Note that in the maximal dynamical rank case, Zhang managed to achieve this by proving that certain LMMP with scaling terminates $G$-equivariantly (see \cite[Proposition~3.11]{Zhang16}).
It is not clear to us that in our setting we can still run a similar $G$-equivariant LMMP with scaling.
The main obstruction is the absence of a nef and big $\bR$-divisor $A$ as essentially constructed in \cite{DS04}, which plays a crucial role in the proof of \cite[ibid.]{Zhang16}.
On the other hand, the induced birational action of $G$ on $X_m$ turns out to be isomorphic in codimension one, i.e., $G|_{X_m}$ is a subgroup of the so-called pseudo-automorphism group $\PsAut(X_m)$ of $X_m$.
It is thus more natural to study the dynamical property of a group $G$ of pseudo-automorphisms of a general $X$.
\item For a normal projective variety $X$, the following is well known:
\begin{equation*}
X \text{ is rationally connected} \Longrightarrow X \text{ is uniruled} \Longrightarrow \kappa(X) = -\infty.
\end{equation*}
However, the implication ``$\kappa(X) = -\infty \Longrightarrow X$ is uniruled" is unknown and turns out to be closely related to one of the most important conjectures in birational geometry, namely, the Non-vanishing conjecture (cf.~\cite[Conjecture~2.1]{BCHM10}; see also \cite[Conjecture~0.1]{BDPP13}).
This is the reason that we assume $X$ to be uniruled in \cref{thmA}(3).
\item Admittedly, the result of our \cref{thmA} does not present a complete characterization due to those technical assumptions. However, using the similar idea, we are able to reduce the general positive Kodaira dimension case to the Kodaira dimension zero case; see \cref{rmk:pos-kappa} for details.
\end{enumerate}
\end{rmk}

%%%%%%%%%%%%%%%%%%%%%%%%%%%%%%%%%%%%%%%%%%%%%%%%%%%%%%%%%%

\section{Preliminaries}
\label{sec-prelim}

%%%%%%%%%%%%%%%%%%%%%%%%%%%%%%%%%%%%%%%%%%%%%%%%%%%%%%%%%%

\noindent
Throughout this section, unless otherwise stated, $X$ is a normal projective variety of dimension $n$ defined over $\bC$.

We refer to Koll\'ar--Mori \cite{KM98} for the standard definitions, notation, and terminologies in birational geometry. For instance, see \cite[Definitions~2.34 and 5.8]{KM98} for the definitions of canonical, Kawamata log terminal (klt), rational, and log canonical (lc) singularities.

The {\it Kodaira dimension} $\kappa(W)$ of a smooth projective variety $W$ is defined as the Kodaira--Iitaka dimension $\kappa(W, K_W)$ of the canonical divisor $K_W$.
The Kodaira dimension of a singular variety is defined to be the Kodaira dimension of any smooth model.

We say that $X$ is {\it uniruled}, if there is a dominant rational map $\bbP_\bC^1 \times Y \ratmap X$ with $\dim Y = n - 1$.
We call $X$ {\it rationally connected}, in the sense of Campana \cite{Campana92} and Koll\'ar--Miyaoka--Mori \cite{KMM92}, if any two general points of $X$ can be connected by an irreducible rational curve on $X$; when $X$ is smooth, this is equivalent to saying that any two points of $X$ can be connected by an irreducible rational curve (see e.g., \cite[IV.3]{Kollar96}).

A fundamental result about rationally connected varieties is arguably the existence of the maximal rationally connected fibration (MRC fibration for short) constructed by Campana \cite{Campana92} and Koll\'ar--Miyaoka--Mori \cite{KMM92}.
Roughly speaking, for any given variety $X$, there exists a dominant rational map $\pi \colon X \ratmap Z$ (unique up to birational equivalence) characterized by the following properties:
\begin{itemize}
\item Rational connectivity: The general fibers of $\pi$ are rationally connected.
\item Maximality: Almost all rational curves in $X$ lie in the fibers. Namely, for a very general point $z\in Z$, if $C$ is a rational curve on $X$ meeting the fiber $X_z$, then $C\subseteq X_z$.
\end{itemize}
The above rational map $\pi$ and the variety $Z$ are unique up to birational equivalence and are called the {\it MRC fibration} and the {\it MRC quotient} of $X$, respectively.
A deep result due to Graber--Harris--Starr asserts that $Z$ is non-uniruled (see \cite[Corollary~1.4]{GHS03}).
Hence, $Z$ is a point if and only if $X$ is rationally connected.
The MRC fibration is particularly useful when our variety $X$ is uniruled but not rationally connected, since in this situation the MRC fibration is a non-trivial rational fibration (with $0<\dim Z<\dim X$).
Later, in \cref{subsec:SMRC}, we will encounter the {\it special} MRC fibration constructed by Nakayama \cite{Nakayama10}.

We now give the formal definition of Q-abelian varieties.

\begin{definition}[{\cite[Definition 2.13]{NZ10}}]
\label{def:Q-abel}
%A morphism $\pi \colon Y \lra X$ of normal projective varieties is called {\it quasi-\'etale}, if it is finite, surjective, and \'etale in codimension one (i.e., there exists a closed subset $Z \subset Y$ of codimension $\ge 2$ such that $\pi|_{Y\setminus Z} \colon Y\setminus Z \lra X$ is \'etale).
A normal projective variety $X$ is called {\it Q-abelian}, if there are an abelian variety $A$ and a finite surjective morphism $A \lra X$ which is \'etale in codimension one.
\end{definition}

In general, given a $G$-action on an algebraic variety $V$, i.e., there is a group homomorphism $G \lra \Aut(V)$, we denote by $G|_V$ the image of $G$ in $\Aut(V)$.
The action of $G$ on $V$ is {\it faithful}, if $G \lra \Aut(V)$ is injective.

Let $G$ be a subgroup of the automorphism group $\Aut(X)$ of $X$.
A rational map $\pi \colon X \ratmap Y$ is called {\it $G$-equivariant} if the $G$-action on $X$ descends to a biregular (possibly non-faithful) action on $Y$.
In other words, for each $g_X \in G$, there is an automorphism $g_Y$ of $Y$ such that $\pi \circ g_X = g_Y \circ \pi$.
We hence denote by $G|_Y$ the image of $G$ in $\Aut(Y)$.

\subsection{Weak decomposition}
\label{subsec:weak-decomp}

The famous Bogomolov--Beauville decomposition theorem asserts that for any compact K\"ahler manifold with numerically trivial canonical bundle, there is a finite \'etale cover that can be decomposed as a product of a torus, Calabi--Yau manifolds, and irreducible holomorphic symplectic manifolds (see \cite{Beauville83}).
Recently, this has been very successfully generalized to normal projective varieties with only klt singularities and numerically trivial canonical divisors by H\"oring and Peternell \cite{HP19}, based on the previous significant work by Druel \cite{Druel18}, Greb, Guenancia, Kebekus and Peternell \cite{GKP16a,GGK19}.
However, in this note, instead of utilizing their strong decomposition theorem, we shall work on a weaker version due to Kawamata \cite{Kawamata85} and developed by Nakayama--Zhang \cite{NZ10};
see \cref{rmk:singular-BB} for a brief explanation.

We begin with the definition of the so-called augmented irregularity.
Note that the irregularity of normal projective varieties is generally not invariant under \'etale in codimension one covers.

\begin{definition}[Augmented irregularity]
\label{def:aug-q}
Let $X$ be a normal projective variety. The {\it irregularity} of $X$ is defined by $q(X) \coloneqq h^1(X, \sO_X)$, where $\sO_X$ stands for the sheaf of rings of regular functions on $X$.
The {\it augmented irregularity} $\widetilde{q}(X)$ of $X$ is defined as the supremum of $q(Y)$ of all normal projective varieties $Y$ with finite surjective morphisms $Y\lra X$, \'etale in codimension one. Namely,
\[
\widetilde{q}(X) \coloneqq \sup \big\{ q(Y) : Y\to X \text{ is finite surjective and \'etale in codimension one} \big\} .
\]
\end{definition}

\begin{remark}
\label{rmk:aug-q}
\begin{enumerate}[(1)]
\item Let $X$ be a normal projective variety with only klt singularities such that $K_X \sim_\bQ 0$.
Then $\widetilde{q}(X) \le \dim X$.
Also, $q(X) = \dim X$ if and only if $X$ is an abelian variety.
It follows that $X$ is Q-abelian if and only if $\widetilde{q}(X) = \dim X$.
See \cite[Proposition~2.10]{NZ10}.
\item The augmented irregularity is invariant under \'etale in codimension one covers.
Namely, if $Y \lra X$ is \'etale in codimension one, then $\widetilde{q}(Y) = \widetilde{q}(X)$.
Clearly, $\widetilde{q}(Y) \le \widetilde{q}(X)$ by the definition.
On the other hand, by the base change any two \'etale in codimension one covers of $X$ is dominated by a third one so that $\widetilde{q}(Y) \ge \widetilde{q}(X)$.
\end{enumerate}
\end{remark}

\begin{definition}[Weak Calabi--Yau variety]
\label{def:wCY}
A normal projective variety $X$ is called a {\it weak Calabi--Yau variety}, if
\begin{itemize}
\item $X$ has only canonical singularities,
\item the canonical divisor $K_X \sim 0$, and
\item the augmented irregularity $\widetilde{q}(X) = 0$.
\end{itemize}
\end{definition}

\begin{remark}
\label{rmk:wCY}
\begin{enumerate}[(1)]
\item Our notion of weak Calabi--Yau may not be standard in the literature, as often for smooth varieties it only requires the irregularity to be zero.
%(see e.g., \cite[section~3.4]{LOP18}).
However, our weak Calabi--Yau varieties appear naturally in the singular Bogomolov--Beauville decomposition of klt varieties with numerically trivial canonical divisors (see e.g., \cite[Theorem~1.3]{GKP16b}).
\item Note that a two-dimensional weak Calabi--Yau variety is exactly a normal projective surface with du Val singularities such that its minimal resolution is a K3 surface and that there is no finite surjective morphism from any abelian surface.
\item It is also worth mentioning that those smooth Calabi--Yau threefolds of quotient type A or K in the sense of \cite{OS01} are, however, not weak Calabi--Yau according to the above definition.
See also \cite[\S14.2]{GGK19}.
It is a natural question whether the topological fundamental group $\pi_1(X)$ of a weak Calabi--Yau variety $X$ is finite; one can also ask a similar question for the \'etale fundamental group $\widehat\pi_1(X_{\reg})$ of the smooth locus $X_{\reg}$ of $X$.
\end{enumerate}
\end{remark}

As we do not treat actual pairs, the variety $X$ being klt is the same as being log terminal. Henceforth, we do not distinguish them.

\begin{lemma}[{cf.~\cite[Lemma 2.12]{NZ10}}]
\label{lemma:alb-closure}
Let $X$ be a normal projective variety with only klt singularities such that $K_X \sim_\bQ 0$.
Then there exists a finite surjective morphism $\tau \colon X^{\alb} \lra X$ satisfying the following conditions, uniquely up to isomorphism over $X$:
\begin{itemize}
\item[(1)] $\tau$ is \'etale in codimension one.
\item[(2)] $\widetilde{q}(X) = q(X^{\alb})$.
\item[(3)] $\tau$ is Galois.
\item[(4)] If $\tau' \colon X' \lra X$ is any finite surjective morphism satisfying the conditions {\em (1)} and {\em (2)}, then there exists a finite surjective morphism $\sigma \colon X' \lra X^{\alb}$, \'etale in codimension one, such that $\tau' = \tau \circ \sigma$.
\end{itemize}
\end{lemma}

The above Galois cover $\tau$ is called the {\it Albanese closure of $X$ in codimension one} by Nakayama and Zhang;
a similar result for smooth projective varieties could be found in \cite{Beauville83}.
Here, the key point is that the universal property allows one to lift the group action to the Albanese closure.

\begin{lemma}[{cf.~\cite[Proposition 3.5]{NZ10}}]
\label{lemma:lifting}
Let $X$ be a normal projective variety with only klt singularities such that $K_X \sim_\bQ 0$, and $f$ an automorphism of $X$.
Then there exist a morphism $\pi \colon \widetilde{X} \lra X$ from a normal projective variety $\widetilde{X}$,
an automorphism $\widetilde{f}$ of $\widetilde{X}$ such that the following conditions hold.
\begin{itemize}
\item[(1)] $\pi$ is finite surjective and \'etale in codimension one.
\item[(2)] $\widetilde{X}$ is isomorphic to the product variety $Z \times A$ for a weak Calabi--Yau variety $Z$ (see \cref{def:wCY}) and an abelian variety $A$.
\item[(3)] The dimension of $A$ equals the augmented irregularity $\widetilde{q}(X)$ of $X$.
\item[(4)] There are automorphisms $\widetilde{f}_Z$ and $\widetilde{f}_A$ of $Z$ and $A$, respectively, such that the following diagram commutes:
\[
\xymatrix{
X \ar[d]_{f} & & \widetilde{X} \ar[ll]_{\pi} \ar[d]_{\widetilde{f}} \ar[rr]^{\isom} & & Z \times A \ar[d]^{\widetilde{f}_Z \times \widetilde{f}_A} \\
X & & \widetilde{X} \ar[ll]_{\pi} \ar[rr]^{\isom} & & Z \times A .
}
\]
%In particular, for any subgroup $G \le \Aut(X)$, the $G$-action on $X$ extends to a faithful $G$-action on $W$, denoted by $G|_W$, then splits as $G|_Z \times G|_A$.
\end{itemize}
\end{lemma}
\begin{proof}
For the convenience of the reader, we sketch their proof as follows.
First, let us take the global index-one cover $X_1 \lra X$, which is a finite surjective morphism and \'etale in codimension one, such that $X_1$ has only canonical singularities with $K_{X_1} \sim 0$ (see \cite[Definition~5.19]{KM98}).
The uniqueness of the global index-one cover asserts that the automorphism $f$ can be lifted to an automorphism $f_1$ on $X_1$.
So at the expense of replacing $(X, f)$ by $(X_1, f_1)$, we may assume that $X$ has only canonical singularities with $K_X \sim 0$.

Next, let $\tau \colon X^{\alb} \lra X$ be the Albanese closure of $X$ in codimension one, whose existence is guaranteed by \cref{lemma:alb-closure}.
It thus follows from the universal property of $\tau$ that we can lift $f$ to an automorphism $f^{\alb}$ on $X^{\alb}$.
More precisely, applying \cref{lemma:alb-closure}(4) to $f \circ \tau$, there exists a finite surjective morphism $f^{\alb} \colon X^{\alb} \lra X^{\alb}$ such that $f \circ \tau = \tau \circ f^{\alb}$.
Clearly, $f^{\alb}$ is an automorphism since so is $f$.
Therefore, replacing $(X, f)$ by $(X^{\alb}, f^{\alb})$ if necessary, we may assume further that $\widetilde{q}(X) = q(X)$.

Note that the augmented irregularity is invariant under \'etale in codimension one covers; see e.g., \cref{rmk:aug-q}(2).
Hence, the above $\widetilde{q}(X)$ is indeed equal to the augmented irregularity of the original $X$, even though we have replaced our $X$ by new models.

Now, under the above assumptions, the Albanese morphism $\alb_X \colon X \lra A \coloneqq \Alb(X)$ turns out to be an \'etale fiber bundle, i.e., there is an isogeny $\phi \colon B \lra A$ such that 
$X \times_A B \isom Z \times B$,
where $Z$ is a fiber of $\alb_X$ (see \cite[Theorem~8.3]{Kawamata85}).
Without loss of generality, we may assume that $\dim A = q(X) > 0$ (for otherwise, $X$ is a weak Calabi--Yau variety).
%It is not hard to verify that $Z$ is a weak Calabi--Yau variety by its definition and also the definition of the augmented irregularity $\widetilde{q}(X)$ (see e.g., the proof of \cite[Proposition~2.10(3)]{NZ10}).
Clearly, there is an induced automorphism of $A$ by the universal property of the Albanese morphism $\alb_X$; denote it by $f_A$.
If $\dim Z = 0$, then $\widetilde{q}(X) = q(X) = \dim A = \dim X$ so that $X$ is an abelian variety (isogenous to $A$); see \cref{rmk:aug-q}(1).
We are also done in this case.
So, let us assume that $0 < \dim Z < \dim X$.
Note that $Z$ has only canonical singularities with $K_Z \sim 0$.
It is not hard to see that $Z$ is a weak Calabi--Yau variety.
Indeed, if $\widetilde{q}(Z) > 0$, then by applying the same argument above to $Z$, there exists a finite surjective morphism $B_0 \times Z_0 \lra Z$ \'etale in codimension one, where $B_0$ is an abelian variety of dimension $\widetilde{q}(Z)>0$.
This gives another finite surjective morphism $B \times B_0 \times Z_0 \lra X$ \'etale in codimension one, from which we have
\[
\widetilde{q}(X) = q(X) = \dim B < \dim B + \dim B_0 \le q(B \times B_0 \times Z_0) \le \widetilde{q}(X),
\]
a contradiction.

Lastly, take an isogeny $\psi \colon A \lra B$ further so that $\phi \circ \psi = [m_A]$ is just the multiplication-by-$m$ map on $A$ for some positive integer $m$.
Then there is an automorphism $\widetilde{f}_A$ of $A$ such that $[m_A] \circ \widetilde{f}_A = f_A \circ [m_A]$.
Consider the new fiber product $\widetilde{X} \coloneqq X \times_A A$ of $\alb_X \colon X \lra A$ and $[m_A] \colon A \lra A$.
Let $\pi \colon \widetilde{X} \lra X$ denote the finite \'etale cover induced from the first projection.
Then $\widetilde{X} \isom Z \times A$ for the same fiber $Z$ of $\alb_X$ as above.
It is clear that those automorphisms $f$, $f_A$ and $\widetilde{f}_A$ induce an automorphism $\widetilde{f}$ on $\widetilde{X}$ satisfying that $\pi \circ \widetilde{f} = f \circ \pi$.
Note that as a weak Calabi--Yau variety, $Z$ is nonruled and has only canonical (and hence rational by \cite[Theorem~5.22]{KM98}) singularities, and its augmented irregularity $\widetilde{q}(Z)$ vanishes.
It thus follows from \cref{lemma:split} below that the induced automorphism of $\widetilde{f}$ on $Z \times A$ splits as $\widetilde{f}_Z \times \widetilde{f}_A$.
In other words, we have the following commutative diagram endowed with equivariant group actions:
\[
\begin{tikzcd}
Z \times A  \arrow[loop above, "\widetilde{f}_Z \times \widetilde{f}_A"]  \ar[rr, "\isom"] \ar[rrdd, bend right=45] & & \widetilde{X} \arrow[loop above, "\widetilde{f}"] \ar[rr] \ar[d] \ar[dd, "\pi"', bend right=70] & & A \arrow[loop, "\widetilde{f}_A", distance=15, out=60, in=30] \ar[d, "\psi"'] \ar[dd, "\text{$[m_A]$}", bend left=70] \\
& & X \times_A B \ar[rr] \ar[d] & & B \ar[d, "\phi"'] \\
& & \arrow[loop, "f"', distance=15, out=210, in=240] X \ar[rr, "\alb_X"] & & A \arrow[loop, "f_A", distance=15, out=-30, in=-60]
\end{tikzcd}.
\]
Finally, in view of the Albanese morphism $\alb_X$, we see that $\dim A = q(X) = \widetilde{q}(X)$.
\end{proof}

\begin{remark}
\label{rmk:Nakayama}
\begin{enumerate}[(1)]
\item In the above lemma, by Nakayama's celebrated result on the Abundance conjecture in the Kodaira dimension zero case (see \cite[Corollary~V.4.9]{Nakayama04}), we can replace the condition ``$K_X \sim_\bQ 0$" by ``$K_X \num 0$".
When $X$ has only canonical singularities, this was originally due to Kawamata \cite[Theorem~8.2]{Kawamata85}.
\item For any subgroup $G \le \Aut(X)$, the action of $G$ on $X$ extends to a faithful action on $\widetilde{X}$, denoted by $\widetilde{G}$, which then splits as a subgroup of $\widetilde{G}|_Z \times \widetilde{G}|_A$ by the following lemma.
Note that the action of $G$ on $X$ can be identified with a not necessarily faithful action of $\widetilde{G}$ on $X$ (with finite kernel).
If $G\isom \bZ^r$ which is always the case in this article, we can apply \cite[Lemma~2.4]{Zhang13} so that a finite-index subgroup of $\widetilde{G}$ also acts faithfully on $X$.
\end{enumerate}
\end{remark}

Below is a simple variant of Nakayama and Zhang's splitting criterion for automorphisms of certain product varieties.

\begin{lemma}[{cf.~\cite[Lemma 2.14]{NZ10}}]
\label{lemma:split}
Let $Z$ be a nonruled normal projective variety with only rational singularities, and $A$ an abelian variety.
Suppose that $q(Z) = 0$.
Then any automorphism $f$ of $Z \times A$ splits, i.e., there are suitable automorphisms $f_Z$ and $f_A$ of $Z$ and $A$, respectively, such that $f = f_Z \times f_A$.
%In particular, $\Aut(Z \times A) \isom \Aut(Z) \times \Aut(A)$.
\end{lemma}

\subsection{Special MRC fibration}
\label{subsec:SMRC}

In this subsection, we collect basic materials on the special MRC fibration introduced by Nakayama \cite{Nakayama10}.

\begin{definition}[Nakayama]
\label{def:sMRC}
Given a projective variety $X$,
a dominant rational map $\pi \colon X \ratmap Z$ is called the {\it special MRC fibration of $X$}, if it satisfies the following conditions:
\begin{itemize}
\item[(1)] The graph $\Gamma_\pi \subseteq X \times Z$ of $\pi$ is equidimensional over $Z$.
\item[(2)] The general fibers of $\Gamma_\pi \lra Z$ are rationally connected.
\item[(3)] $Z$ is a non-uniruled normal projective variety (see \cite{GHS03}).
\item[(4)] If $\pi' \colon X \ratmap Z'$ is a dominant rational map satisfying (1)--(3), then there is a birational morphism $\nu \colon Z' \lra Z$ such that $\pi = \nu \circ \pi'$.
\end{itemize}
\end{definition}

The existence and the uniqueness (up to isomorphism) of the special MRC fibration is proved in \cite[Theorem~4.18]{Nakayama10}.
One of the crucial advantages of the special MRC is the following descent property (see \cite[Theorem~4.19]{Nakayama10}).

\begin{lemma}
\label{lemma:sMRC}
Let $\pi \colon X \ratmap Z$ be the special MRC fibration, and $G \le \Aut(X)$.
Then $G$ descends to a biregular action on $Z$, denoted by $G|_Z$.
Moreover, there exist a birational morphism $p \colon W \lra X$ and an equidimensional surjective morphism $q \colon W \lra Z$ satisfying the following conditions:
\begin{itemize}
\item[(1)] $W$ is a normal projective variety.
\item[(2)] A general fiber of $q$ is rationally connected.
\item[(3)] Both $p$ and $q$ are $G$-equivariant.
\end{itemize}
%We hence have the following commutative diagram of normal projective varieties:
%\[\xymatrix{
%& & W \ar[lld]_{p} \ar[rrd]^{q} & & \\
%X \ar@{-->}[rrrr]^{\pi} & & & & Z \\ }
%\]
\end{lemma}
\begin{proof}
By \cite[Theorem~4.19]{Nakayama10}, $G$ descends to a biregular action on $Z$.
We take $W$ as the normalization of the graph $\Gamma_\pi$ of $\pi$ which admits a natural faithful $G$-action.
Then (2) follows readily from \cref{def:sMRC}, while (3) the $G$-equivariance of $\pi$.
\end{proof}

\begin{lemma}[{cf.~\cite[Lemma 4.4]{NZ10}}]
\label{lemma:sMRC-cover}
With notation as in \cref{lemma:sMRC}, let $\theta_Z \colon Z' \lra Z$ be a $G|_Z$-equivariant finite surjective morphism from a normal projective variety $Z'$.
Then there exist finite surjective morphisms $\theta_X \colon X' \lra X$ and $\theta_W \colon W' \lra W$,
a birational morphism $p' \colon W' \lra X'$,
and an equidimensional surjective morphism $q' \colon W' \lra Z'$ satisfying the following conditions:
\begin{itemize}
\item[(1)] Both $X'$ and $W'$ are normal projective varieties.
\item[(2)] A general fiber of $q'$ is rationally connected.
\item[(3)] $\pi' \coloneqq q' \circ p'^{-1}$ is $G$-equivariantly birational to the MRC fibration of $X'$.
\item[(4)] In the commutative diagram below, every morphism or rational map other than $\theta_Z$ is $G$-equivariant.
\end{itemize}
\[
\xymatrix{
& & W' \ar[lld] _{p'} \ar[rrd]^{q'} \ar[dd]_(.3){\theta_W} & & \\
X' \ar@{-->}[rrrr]^(.6){\pi'} \ar[dd]_{\theta_X} & & & & Z' \ar[dd]^{\theta_Z} \\
& & W \ar[lld]_{p} \ar[rrd]^{q} & & \\
X \ar@{-->}[rrrr]^{\pi} & & & & Z \\
}
\]
Moreover, if $\theta_Z$ is \'etale in codimension one, then so are $\theta_X$ and $\theta_W$.
\end{lemma}
\begin{proof}
Let $W'$ be the normalization of the fiber product $W \times_Z Z'$.
Denote by $\theta_W \colon W' \lra W$ and $q' \colon W' \lra Z'$ the morphisms induced from the first and second projections, respectively.
Then $q'$ is an equidimensional surjective morphism whose general fibers are rationally connected varieties and in particular irreducible, since so is $q$.
Here we use the fact that smooth rationally connected varieties are simply connected.
This forces $W'$ to be irreducible and hence $W'$ is a normal projective variety.
Clearly, the $G$-actions on $W$ and $Z'$ can be naturally extended to $W \times_Z Z'$ and hence to $W'$, which is faithful since $G$ acts faithfully on $W$.
Note that $Z'$ is non-uniruled since so is $Z$.
It follows that $q'$ is $G$-equivariantly birational to the special MRC fibration of $W'$ by \cref{def:sMRC}.
Taking the Stein factorization of the composite $p \circ \theta_W \colon W' \lra W \lra X$,
we then have a birational morphism $p' \colon W' \lra X'$ and a finite morphism $\theta_X \colon X' \lra X$ for a normal projective variety $X'$ such that $p \circ \theta_W = \theta_X \circ p'$;
furthermore, the faithful $G$-actions on $W'$ and $X$ also induce a faithful $G$-action on $X'$.
%(see e.g., \cite[\href{https://stacks.math.columbia.edu/tag/03H0}{Theorem 03H0} and \href{https://stacks.math.columbia.edu/tag/035J}{Lemma 035J}]{stacks-project}).
Since the notion of the MRC fibration is essentially birational in nature, $\pi' = q' \circ p'^{-1}$ is also $G$-equivariantly birational to the MRC fibration of $X'$.
So all conditions (1)--(4) have been satisfied.

The last part follows from the fact that being \'etale is a local property stable under base change.
\end{proof}

%\begin{question}
%Is the above $\pi'$ isomorphic to the special MRC fibration of $X'$?
%\end{question}

\subsection{Dynamical ranks}
\label{subsec:dyn-rank}

In this section, we shall consider the dynamical rank of group actions in a much more general setting.
We first recall the following Tits alternative type theorem due to De-Qi Zhang \cite{Zhang09}.

\begin{theorem}[{cf.~\cite{Zhang09}}]
\label{thm:Zhang-A}
Let $X$ be a normal projective variety of dimension $n \ge 2$ and $G$ a subgroup of $\Aut(X)$.
Then one of the following two assertions holds.
\begin{itemize}
\item[(1)] $G$ contains a subgroup isomorphic to the non-abelian free group $\bZ*\bZ$.
\item[(2)] There is a finite-index subgroup $G_1$ of $G$ such that the induced group $G_1|_{\NS(X)_\bR}$ is solvable and Z-connected. Moreover, the null-entropy subset $N(G_1)$ of $G_1$ is a normal subgroup of $G_1$ and the quotient group $G_1/N(G_1)$ is free abelian of rank $r \le n-1$.
\end{itemize}
\end{theorem}

\begin{rmk}
\label{rmk:dyn-rank}
In general, the induced group $G|_{\NS(X)_\bR}$ of $G$ is called {\it Z-connected} if its Zariski closure in $\GL(\NS(X)_\bC)$ is connected with respect to the Zariski topology.
Note that being Z-connected is only a technical condition for us to apply the theorem of Lie--Kolchin type for a cone in \cite{KOZ09}.
Actually, it is always satisfied by replacing the group with a finite-index subgroup (see e.g., \cite[Remark~3.10]{DHZ15}).
We will frequently use this fact without mentioning it very precisely.
\end{rmk}

We also remark that in the second assertion of the above \cref{thm:Zhang-A}, the rank of $G_1/N(G_1)$ is independent of the choice of $G_1$.
Hence, it makes sense to think of this as an invariant of $G$.
We introduce the following notion of dynamical rank in a much broader sense.

\begin{definition}[Dynamical rank]
\label{def:dyn-rank}
Let $X$ be a normal projective variety of dimension $n$ and $G$ a subgroup of $\Aut(X)$ such that $G|_{\NS(X)_\bR}$ is solvable.
Then the rank of the free abelian group $G/N(G)$ is called the {\it dynamical rank of $G$}, and denoted by $\dr(G)$.
\end{definition}

As one may have noticed, we suppress the condition ``$G|_{\NS(X)_\bR}$ is Z-connected".
This does not affect the well-definedness of our dynamical rank according to \cref{rmk:dyn-rank}.

Sometimes, we may write $\dr(G|_X)$ to emphasize that it is the dynamical rank of the group $G$ acting on $X$.
Conventionally, the dynamical rank of a group of null entropy is always zero.
We first quote the following result which generalizes \cite[Lemma~2.10]{Zhang09}.

\begin{lemma}[{cf.~\cite[Lemmas 4.1 and 4.3]{Hu19}}]
\label{Zha-Hu1}
Let $\pi \colon X \ratmap Y$ be a $G$-equivariant dominant rational map of normal projective varieties with $n = \dim X > \dim Y = m > 0$.
Suppose that $G|_{\NS(X)_\bR}$ is solvable.
Then so is $G|_{\NS(Y)_\bR}$, and we have
\[
\dr(G|_X)\le \dr(G|_Y) + n - m - 1.
\]
In particular, $\dr(G|_X) = n - 2$ only if $\dr(G|_Y) = m - 1$.
\end{lemma}

The lemma below asserts that our dynamical rank is actually a birational invariant.
See also \cite[Lemma~3.1]{Zhang16} for a similar treatment.

\begin{lemma}[{cf.~\cite[Lemmas 4.2 and 4.4]{Hu19}}]
\label{Zha-Hu2}
Let $\pi \colon X\dashrightarrow Y$ be a $G$-equivariant generically finite dominant rational map of normal projective varieties.
Then after replacing $G$ by a finite-index subgroup, $G|_{\NS(X)_\bR}$ is solvable if and only if so is $G|_{\NS(Y)_\bR}$.
Moreover, $\dr(G|_X) = \dr(G|_Y)$.
\end{lemma}

%%%%%%%%%%%%%%%%%%%%%%%%%%%%%%%%%%%%%%%%%%%%%%%%%%%%%%%%%%

\section{Proof of Theorem \ref{thmA}}
\label{sec-proof}

%%%%%%%%%%%%%%%%%%%%%%%%%%%%%%%%%%%%%%%%%%%%%%%%%%%%%%%%%%

\noindent
The theorem will follow immediately from the following lemmas.
Each one will correspond to one assertion of \cref{thmA}.

\begin{lemma}
\label{lemma:pos-kappa}
Let $(X,G)$ satisfy \hyperref[hyp]{\rm Hyp$(n,n-2)$} with $n \ge 3$.
Suppose that the Kodaira dimension $\kappa(X)$ of $X$ is positive.
Then $\kappa(X) = 1$ and there exists a dominant rational fibration $\phi \colon X \ratmap B$ for some curve $B$ such that after replacing $G$ by a finite-index subgroup, the following assertions hold.
\begin{itemize}
\item[(1)] $G$ descends to a trivial action on the base curve $B$ of $\phi$.
\item[(2)] Let $F$ be a very general fiber $F$ of $\phi$. Then the induced $G$-action on $F$ is faithful such that the pair $(F,G|_F)$ satisfies \hyperref[hyp]{\rm Hyp$(n-1,n-2)$}. Moreover, $F$ is $G$-equivariantly birational to a K3 surface, an Enriques surface, or a Q-abelian variety (see \cref{def:Q-abel}).
\end{itemize}
\end{lemma}
\begin{proof}
Let $\phi \coloneqq \Phi_{|mK_X|} \colon X \ratmap B \subseteq \bbP(H^0(X, mK_X))$ be the Iitaka fibration of $X$ with $B$ the image of $\Phi_{|mK_X|}$ for $m \gg 0$.
It follows from the Deligne--Nakamura--Ueno theorem (see \cite[Theorem~14.10]{Ueno75}) that $G$ descends to a finite group $G|_B$ acting on $B$ biregularly.
Replacing $G$ by a finite-index subgroup, which does not change its dynamical rank, we may assume that $G|_B = \{ \id \}$.
Further, replacing $X$ and $B$ by $G$-equivariant resolutions of singularities of the graph $\Gamma_\phi$ of $\phi$ and of $B$, we may also assume that $\phi$ is a regular morphism and $B$ is smooth,
since by \cref{Zha-Hu2} the new pair $(X,G)$ still satisfies \hyperref[hyp]{\rm Hyp$(n,n-2)$}.
If $\kappa(X) = n$, i.e., $X \ratmap B$ is birational, then again thanks to \cref{Zha-Hu2}, we have $n-2 = \dr(G|_X) = \dr(G|_B) = 0$, a contradiction.
So we may assume that $0 < \kappa(X) < n$, which yields that $\phi \colon X \lra B$ is a non-trivial $G$-equivariant fibration.
It thus follows from \cref{Zha-Hu1} that
\[
n-2 = \dr(G|_X) \le \dr(G|_B) + n - \dim B - 1 = n - \kappa(X) - 1,
\]
and hence $\kappa(X) = 1$ so that $B$ is a curve.

It remains to show the assertion (2).
Since $G$ acts trivially on the base, $G$ acts naturally and regularly on the very general fiber $F$ of $\phi$.
For any $g \in G$, let $g_F$ denote the induced automorphism of $g$ on $F$.
By the product formula (see \cite[Theorem~1.1]{DN11}), the first dynamical degree $d_1(g_F)$ of $g_F$ equals $d_1(g)$ which is larger than $1$ if $g \ne \id$.
Therefore, $G$ acts faithfully (and also regularly) on $F$ so that we can identify $G$ with $G|_F \le \Aut(F)$ and $(F,G|_F)$ satisfies \hyperref[hyp]{\rm Hyp$(n-1,n-2)$}.

Lastly, note that $F$, as a very general fiber of the Iitaka fibration, has Kodaira dimension zero and hence is not rationally connected.
Then Zhang's \cref{thm:Zhang} yields that, up to replacing $G$ by a finite-index subgroup, $F$ is $G$-equivariantly birational to a Q-abelian variety if $\dim F \ge 3$ or equivalently $n \ge 4$.
On the other hand, if $\dim F = 2$, since it admits an automorphism of positive entropy, it is well known that our $F$ is either a K3 surface, an Enriques surface, or an abelian surface (see \cite[Proposition~1]{Cantat99}).
\end{proof}

\begin{remark}
\label{rmk:pos-kappa}
Using a similar proof of the above lemma, one can also show the following result.
Let $(X,G)$ satisfy \hyperref[hyp]{\rm Hyp$(n,r)$} with $1 \le r \le n-2$.
If the Kodaira dimension $\kappa(X)$ of $X$ is positive, then $\kappa(X) \le n-r-1$ (this is actually not new; see \cite[Lemma~2.11]{Zhang09}).
Moreover, after replacing $G$ by a finite-index subgroup, we may assume that $G$ acts trivially on $B$ and naturally on the very general fiber $F$ of the Iitaka fibration $\phi \colon X \ratmap B$ with $\dim F = n-\kappa(X)$.
Better still, the product formula asserts that for each $g \in G \setminus \{\id\}$, the restriction $g_F$ of $F$ is of positive entropy since so is $g$.
Hence, the $G$-action on $F$ is faithful and $(F, G|_F)$ satisfies \hyperref[hyp]{\rm Hyp$(n-\kappa(X), r)$}.
In summary, we have the following reduction:
\[
\begin{tikzcd}
\text{\hyperref[hyp]{\rm Hyp$(n,r)$}} \text{ with } \kappa>0 \arrow[r, rightsquigarrow] & \text{\hyperref[hyp]{\rm Hyp$(n',r)$}} \text{ with } \kappa=0 \text{ and } n'<n.
\end{tikzcd}
\]
\end{remark}

The following lemma partially deals with the Kodaira dimension zero case.

\begin{lemma}
\label{lemma:kappa=0}
Let $(X,G)$ satisfy \hyperref[hyp]{\rm Hyp$(n,n-2)$} with $n \ge 3$.
Suppose that $X$ has only klt singularities and $K_X \num 0$.
Then after replacing $G$ by a finite-index subgroup, there exist a finite cover $Y \lra X$, \'etale in codimension one, and a faithful $G$-action on $Y$ such that $Y$ is $G$-equivariantly birational to one of the following varieties:
\begin{itemize}
\item[(1)] an abelian variety $A$, where $(A, G|_A)$ satisfies \hyperref[hyp]{\rm Hyp$(n,n-2)$};
\item[(2)] a weak Calabi--Yau variety $Z$, where $(Z, G|_Z)$ satisfies \hyperref[hyp]{\rm Hyp$(n,n-2)$};
\item[(3)] a product of a weak Calabi--Yau surface $S$ and an abelian variety $A$, where $(S, G|_S)$ and $(A, G|_A)$ satisfy \hyperref[hyp]{\rm Hyp$(2,1)$} and \hyperref[hyp]{\rm Hyp$(n-2,n-3)$}, respectively.
\end{itemize}
\end{lemma}
\begin{proof}
It follows from \cref{lemma:lifting,rmk:Nakayama} that there is a finite cover $\widetilde{X} \lra X$, \'etale in codimension one, such that $\widetilde{X} \isom Z \times A$ for a weak Calabi--Yau variety $Z$ and an abelian variety $A$ of dimension $\widetilde{q}(X)$, the augmented irregularity of $X$;
furthermore, the action of $G$ on $X$ extends to a faithful action of $\widetilde{G}$ on $\widetilde{X}$.
Replacing $\widetilde{G}$ by a finite-index subgroup, we may assume that $\widetilde{G}$ also acts faithfully on $X$ and can be identified with a finite-index subgroup of $G$ (cf.~\cite[Lemma~2.4]{Zhang13}).
Therefore, after replacing $G$ by the above mentioned finite-index subgroup, we may assume that $(\widetilde{X}, G)$ satisfies \hyperref[hyp]{\rm Hyp$(n,n-2)$} by \cref{Zha-Hu2} since so does $(X,G)$.
We hence have the following three cases to analyze.

{\it Case $1$.} $\widetilde{q}(X) = n$ and hence $\widetilde{X} = A$ is an abelian variety. In this case, the pair $(A, G|_A)$ satisfies \hyperref[hyp]{\rm Hyp$(n,n-2)$} and we just take $Y$ to be $A$.

{\it Case $2$.} $\widetilde{q}(X) = 0$ and hence $\widetilde{X} = Z$ is a weak Calabi--Yau variety of dimension $n$. So the pair $(Z, G|_Z)$ also satisfies \hyperref[hyp]{\rm Hyp$(n,n-2)$}.
We then choose $Y$ to be $Z$.

{\it Case $3$.} $0 < \widetilde{q}(X) < n$ so that $\widetilde{X}$ is an actual product $Z \times A$ with each factor being positive-dimensional.
According to \cref{lemma:split}, we denote by $G|_Z$ and $G|_A$ the induced group actions of $G$ on $Z$ and $A$, respectively; note that both are finitely generated abelian groups.
It follows from \cref{Zha-Hu1} that $\dr(G|_Z) = \dim Z - 1 \eqqcolon r_1$ and $\dr(G|_A) = \dim A - 1 \eqqcolon r_2$.
Applying \cite[Theorem~I]{DS04} to the pair $(A, G|_A)$ yields that the null-entropy subgroup $N(G|_A)$ of $G|_A$ is finite.
So, up to replacing $G|_A$ and hence $G$ by a finite-index subgroup, we may assume that $G|_A\isom \bZ^{r_2}$ is a free abelian group of positive entropy.
Thanks to \cref{Zha-Hu2}, the same argument applies to the $G|_Z$-equivariant resolution of $Z$.
Thus we can assume that $G|_Z \isom \bZ^{r_1}$ is of positive entropy.
In particular, $(Z, G|_Z)$ and $(A, G|_A)$ satisfy \hyperref[hyp]{\rm Hyp$(r_1+1,r_1)$} and \hyperref[hyp]{\rm Hyp$(r_2+1,r_2)$}, respectively.

If $\dim Z = 2$ (i.e., $r_1=1$), then $Z$ is just a weak Calabi--Yau surface $S$.
So in this case we take $Y$ to be $\widetilde{X} \isom S\times A$.

Let us consider the case when $\dim Z \ge 3$.
Recall that as a weak Calabi--Yau variety (see \cref{def:wCY}), $Z$ is not rationally connected and has only canonical singularities with $K_Z \sim 0$.
So applying \cref{thm:Zhang} to $(Z, G|_Z \isom \bZ^{r_1})$ asserts that, up to replacing $G|_Z$ and hence $G$ by a finite-index subgroup, $Z$ is birational to a Q-abelian variety $B/F$ such that the induced action of $G|_Z$ on $B/F$ is biregular, where $B$ is an abelian variety and $F$ is a finite group whose action on $B$ is free outside a finite subset of $B$;
moreover, there is a faithful action of $G|_Z$ on $B$ such that $B \lra B/F$ is $G|_Z$-equivariant.
Clearly, the pair $(B, G|_B = G|_Z)$ satisfies \hyperref[hyp]{\rm Hyp$(r_1+1,r_1)$}  since so does $(Z, G|_Z)$.
Let $\widetilde{Z}$ be the normalization of the fiber product $Z \times_{B/F} B$, which inherits a natural faithful $G|_Z$-action.
Then the induced projection $\widetilde{Z} \lra Z$ is finite surjective and \'etale in codimension one.
Also, $\widetilde{Z} \ratmap B$ is a $G|_Z$-equivariant birational map.
This yields that $Y \coloneqq \widetilde{Z} \times A$ is $G$-equivariantly birational to the abelian variety $B \times A$,
while $Y \lra \widetilde{X} \lra X$ is still \'etale in codimension one.
It is easy to see that $(B \times A, G = G|_B \times G|_A)$ also satisfies \hyperref[hyp]{\rm Hyp$(n,n-2)$}.
We thus complete the proof of \cref{lemma:kappa=0}.
\end{proof}

\begin{remark}
\label{rmk:kappa=0}
If $X$ is smooth, we are able to give a finer characterization as follows.
Recall that for a projective manifold $X$ with numerically trivial canonical bundle, there exists a unique minimal splitting cover $\widetilde{X}$ in the sense of Beauville \cite[\S3]{Beauville83}, of the form
\[
A \times \prod V_i \times \prod X_j,
\]
where $A$ is an abelian variety, the $V_i$ are (simply connected) Calabi--Yau manifolds and the $X_j$ are projective hyper-K\"ahler manifolds.

As a consequence, any automorphism of $X$ extends to $\widetilde{X}$ and then splits into pieces (up to permutations).
More precisely, if $G \isom \bZ^{n-2}$ is a subgroup of $\Aut(X)$ such that $G$ is of positive entropy,
then there exists a group $\widetilde{G}$ (the lifting of $G$) acting faithfully on $\widetilde{X}$ such that $G=\widetilde{G}/F$, where $F$ is the Galois group of the minimal splitting cover $\widetilde{X} \lra X$.
Replacing $\widetilde{G}$ by a finite-index subgroup, we may assume that $\widetilde{G}$ also acts faithfully on $X$ (cf.~\cite[Lemma~2.4]{Zhang13}), both $(\widetilde{X}, \widetilde{G})$ and $(X, \widetilde{G})$ satisfy \hyperref[hyp]{\rm Hyp$(n,n-2)$};
further, the group $\widetilde{G}$ acting on $\widetilde{X}$ splits as a subgroup of
\[
\widetilde{G}|_A \times \prod \widetilde{G}|_{V_i} \times \prod \widetilde{G}|_{X_j}.
\]

One can use the similar argument as in \cref{lemma:kappa=0} to show that there are at most two factors.
Moreover, it is well-known that $\dr(\widetilde{G}|_{X_j}) \le 1$ (see e.g., \cite[Theorem~4.6]{KOZ09}) so that the $X_j$ are K3 surfaces.
In summary, the covering space $\widetilde{X}$ decomposes into a product of abelian varieties, Calabi--Yau manifolds, or K3 surfaces with at most two factors.
Clearly, there are seven possibilities/classes.

%We believe that to deal with the general case when $G \isom \bZ^r$ is then more or less a combinatorial problem.
\end{remark}

\begin{remark}
\label{rmk:singular-BB}
Unfortunately, we are not able to deal with the singular case in an analogous way as in \cref{rmk:kappa=0}, though we already have the Bogomolov--Beauville decomposition for minimal models with trivial canonical class due to H\"oring and Peternell \cite[Theorem~1.5]{HP19}.
The reason for this is as follows.

Let $X$ be a normal projective variety with at most klt singularities such that $K_X \num 0$.
Let $\pi \colon \widetilde{X} \lra X$ be a finite cover, \'etale in codimension one, such that
\[
\widetilde{X} \isom A \times \prod Y_j \times \prod Z_k,
\]
where $A$ is an abelian variety, the $Y_j$ are (singular) Calabi--Yau varieties and the $Z_k$ are (singular) irreducible holomorphic symplectic varieties (see \cite[Definition~1.3]{GGK19}).
Note that a compact K\"ahler manifold with numerically trivial canonical bundle has an almost abelian (aka abelian-by-finite) fundamental group.
This fact is used to conclude the existence of the unique minimal splitting cover in \cite[\S3]{Beauville83} for the smooth case.
However, in the general singular setup, as far as we can tell, the finiteness of fundamental groups of Calabi--Yau varieties is still unknown (see e.g., \cite[\S13]{GGK19}).
It is thus not clear to us that we can always lift the automorphisms of $X$ to some splitting cover $\widetilde{X}$.
The failure of the strategy of \cref{rmk:kappa=0} for general singular varieties forces us to work on the weak decomposition as we mentioned earlier at the beginning of \cref{subsec:weak-decomp}.
\end{remark}

Finally, it remains to consider the negative Kodaira dimension case, where the existence of the so-called special MRC fibration plays a crucial role (see \cref{subsec:SMRC}, or rather \cite[Theorem~4.18]{Nakayama10}).

\begin{lemma}
\label{lemma:neg-kappa}
Let $(X,G)$ satisfy \hyperref[hyp]{\rm Hyp$(n,n-2)$} with $n \ge 3$.
Suppose that $X$ is uniruled but not rationally connected.
Let $\pi \colon X \ratmap Z$ be the special MRC fibration of $X$.
Then one of the following assertions holds.
\begin{itemize}
\item[(1)] $Z$ is a curve of genus $\ge 1$.
\item[(2)] $Z$ is a K3 surface, an Enriques surface, or an abelian surface such that $\dr(G|_Z) = 1$.
\item[(3)] $Z$ has dimension at least $3$. Then after replacing $G$ by a finite-index subgroup, $Z$ is birational to a Q-abelian variety $A/F$ such that the induced action of $G|_Z$ on $A/F$ is biregular, where $A$ is an abelian variety and $F$ is a finite group acting on $A$ freely outside a finite subset of $A$;
moreover, there is a faithful action of $G|_Z$ on $A$ such that the quotient map $A \lra A/F$ is $G|_Z$-equivariant,
and hence by \cref{lemma:sMRC-cover} there exists a finite cover $X' \lra X$, \'etale in codimension one, such that the induced map $\pi' \colon X' \ratmap A$ is $G$-equivariantly birational to the MRC fibration of $X'$.
%\warning{The $X'$ is clear}
\end{itemize}
\end{lemma}
\begin{proof}
Note that $Z$ has dimension at least one because $X$ is not rationally connected.
By \cref{lemma:sMRC} or \cite[Theorem~4.19]{Nakayama10}, $G$ descends to a biregular action $G|_Z$ on $Z$.
Since $Z$ is non-uniruled (see \cite{GHS03}), $\pi$ is a non-trivial $G$-equivariant rational fibration.
It follows from \cref{Zha-Hu1} that $\dr(G|_Z) = \dim Z-1$.

Note that $Z$ is not rationally connected since it is non-uniruled.
Therefore, if $\dim Z = 1$, then $Z$ is a curve of genus $\ge 1$.
If $\dim Z = 2$, then $Z$ is either a K3 surface, an Enriques surface or an abelian surface (see e.g., \cite[Proposition~1]{Cantat99}).
If $\dim Z\ge 3$, similar as in the proof of \cref{lemma:kappa=0}, Case 3, up to replacing $G|_Z$ and hence $G$ by a finite-index subgroup, we may assume that $(Z,G|_Z)$ satisfies \hyperref[hyp]{\rm Hyp$(\dim Z, \dim Z - 1)$} so that \cref{thm:Zhang} applies to $(Z,G|_Z)$.
More precisely, $Z$ is birational to a Q-abelian variety $A/F$ such that the induced action of $G|_Z$ on $A/F$ is biregular, where $A$ is an abelian variety and $F$ is a finite group whose action on $A$ is free outside a finite subset of $A$;
moreover, the $G|_Z$-action on $A/F$ extends to a faithful action on $A$ such that $A \lra A/F$ is also $G|_Z$-equivariant.
Now, by \cref{lemma:sMRC-cover}, there exist a normal projective variety $X'$ and a finite cover $X' \lra X$, \'etale in codimension one, such that the induced map $\pi' \colon X' \ratmap A$ is $G$-equivariantly birational to the MRC fibration of $X'$.
\end{proof}

\begin{proof}[Proof of \cref{thmA}]
It follows from \cref{lemma:pos-kappa,lemma:kappa=0,lemma:neg-kappa}.
\end{proof}

%\begin{rmk} \label{rmkB}
%In \cref{thmA}, if we replace the assumption ``Let $(X,G)$ satisfy \hyperref[hyp]{\rm Hyp$(n,n-2)$} with $n\ge 3$" by a weaker one ``Let $X$ be a normal projective variety of dimension $n\ge 3$ and $G$ a subgroup of $\Aut(X)$, such that $G|_{\NS(X)_\bR}$ is solvable and the dynamical rank $\dr(G)$ of $G$ is $n-2$", then almost all conclusions still hold.
%\end{rmk}

%%%%%%%%%%%%%%%%%%%%%%%%%%%%%%%%%%%%%%%%%%%%%%%%%%%%%%%%%%

\section*{Acknowledgments}

\noindent
The authors would like to thank De-Qi Zhang for many stimulating discussions and his helpful comments on an earlier draft.
The first author is much obliged to St\'ephane Druel, Stefan Kebekus, Thomas Peternell and Chenyang Xu for answering his questions on the finiteness of fundamental groups of singular Calabi--Yau varieties.
The authors are also very grateful to the referee for carefully reading the manuscript and for his/her many helpful comments and suggestions.

%%%%%%%%%%%%%%%%%%%%%%%%%%%%%%%%%%%%%%%%%%%%%%%%%%%%%%%%%%

%\linespread{1.1}

%\bibliographystyle{amsplain}
%\bibliographystyle{amsalpha}
%\bibliographystyle{siam}
%\bibliographystyle{alpha}
%\bibliographystyle{plain}
%\bibliographystyle{unsrt}
%\bibliographystyle{abbrv}
%\bibliographystyle{acm}
%\bibliography{../mybib}

\providecommand{\bysame}{\leavevmode\hbox to3em{\hrulefill}\thinspace}
\providecommand{\MR}{\relax\ifhmode\unskip\space\fi MR }
\providecommand{\MRhref}[2]{%
  \href{http://www.ams.org/mathscinet-getitem?mr=#1}{#2}
}
\providecommand{\href}[2]{#2}

\end{document}